\documentclass[12pt,a4paper,oneside,reqno,notitlepage]{amsart}
\usepackage{amssymb}

\usepackage{amsfonts}
\usepackage{amscd,amsmath}
\usepackage{euscript}
\usepackage{comment}

\usepackage[arrow,matrix,curve]{xy}\SilentMatrices
\def\xyma{\xymatrix@M.7em}

\numberwithin{equation}{section}

\usepackage[T2A]{fontenc}

\newtheorem{prop}{Proposition}[section]
\newtheorem{theorem}{Theorem}[section]
\newtheorem{lemma}{Lemma}[section]
\newtheorem{remark}{Remark}[section]

\def\Ga{\Gamma}

\def\la{\longrightarrow}
\def\ot{\otimes}

\def\bee{\begin{equation}}
\def\ee{\end{equation}}

\def\tor{\mathrm{Tor}}
\def\Tor{\mathrm{Tor}}

\begin{document}
\title{On the homology of the dual de Rham complex}
\author{Roman Mikhailov}
\begin{abstract}
We study the homology of the dual de Rham complex as functors on
the category of abelian groups. We give a description of homology
of the dual de Rham complex up to degree 7 for free abelian groups
and present a corrected version of the proof of Jean's
computations of the zeroth homology group.
\end{abstract}
\maketitle

\subsection{Divided power functor} Let $\sf Ab$ be the category of abelian groups. Recall
the definition of the graded divided power functor (see
\cite{roby}) $\Gamma_*=\bigoplus_{n\geq 0}\Gamma_n: \sf Ab\to \sf
Ab$. The graded abelian group $\Gamma_\ast(A)$ is generated by
symbols $\gamma_i(x)$ of degree $i\geq 0$ satisfying the following
relations for all $x,y \in A$:
\begin{align*} & 1)\ \gamma_0(x) = 1 \\ & 2)\ \gamma_1(x)=x\\ &
3)\
\gamma_s(x)\gamma_t(x)=\binom{s+t}{s}\gamma_{s+t}(x)\\
& 4)\ \gamma_n(x+y)=\sum_{s+t=n}\gamma_s(x)\gamma_t(y),\ n\geq 1\\
& 5)\ \gamma_n(-x)=(-1)^n\gamma_n(x),\ n\geq 1.
\end{align*}
In particular, the canonical map $ A \simeq \Gamma_1(A)$ is an
isomorphism.  The following additional properties of elements of
the abelian group $\Gamma(A)$ will be useful ($x,y\in A,\ r\geq
1$):
\begin{align*}
& \gamma_r(nx)=n^r\gamma_r(x),\ n\in \mathbb Z;\\
& r\gamma_r(x)=x\gamma_{r-1}(x);\\
& x^r=r!\gamma_r(x);\\
& \gamma_r(x)y^r=x^r\gamma_r(y).
\end{align*}
A direct computation implies that
$$
\Gamma_r(\mathbb Z/n)\simeq \mathbb Z/{n(r,n^\infty)},
$$
where $(r,n^\infty)$ is the limit $\lim_{m\to \infty}(r,n^m)$. The
degree 2 component $\Ga_2(A)$ of the divided power algebra is the
Whitehead functor $\Ga (A)$. It is the universal group for
homogenous
 quadratic maps from $A$ into abelian groups.

\subsection{Dual de Rham complex} Let $A$ be an abelian group.
For $n\geq 1,$ denote by $SP^n$ and $\Lambda^n$ the $n$th
symmetric and exterior power functors respectively. For $n\geq 1$,
let $D_*^n(A)$ and $C_*^n(A)$ be the complexes of abelian groups
defined by
\begin{align*}
& D_i^n(A)=SP^i(A)\otimes \Lambda^{n-i}(A),\ 0\leq i\leq n,\\
& C_i^n(A)=\Lambda^i(A)\otimes \Gamma_{n-i}(A),\ 0\leq i\leq n,
\end{align*}
where the differentials $d^i: D_i^n(A)\to D_{i-1}^n(A)$ and $d_i:
C_i^n(A)\to C_{i-1}^n(A)$ are:
\begin{align*}
& d^i((b_1\dots b_i)\otimes b_{i+1}\wedge \dots \wedge b_n)=
\sum_{k=1}^i(b_1\dots \hat b_k \dots b_i)\otimes b_k\wedge
b_{i+1}\wedge\dots \wedge b_n\\
& d_i(b_1\wedge \dots \wedge b_i\otimes X)=\sum_{k=1}^i(-1)^k
b_1\wedge \dots \wedge\hat b_k\wedge \dots \wedge b_i\otimes b_kX
\end{align*}
for any $X\in \Gamma_{n-i}(A)$. The complex $D^n(A)$ is the degree
$n$ component of the classical de Rham complex, first introduced
in the present context of polynomial functors in \cite{FLS} and
denoted $\Omega_n$ in \cite{Franjou}. The dual complexes  $C^n(A)$
were considered in \cite{Jean}. We will call them  the dual de
Rham complexes.

The dual de Rham complexes appear naturally in the theory of
homology of Eilenberg-Mac Lane spaces. Let $A$ be a free abelian
group. There are well-known natural isomorphisms (see, for
example, \cite{Breen}):
\begin{align*}
& H_nK(A,1)= \Lambda^n(A),\ n\geq 1\\
& H_{2n}K(A,2)= \Gamma_n(A),\ n\geq 1\\
& H_{2n+1}K(A,2)=0,\ n\geq 0.
\end{align*}
Consider the path-fibration:
$$
K(A,2)\to PK(A,3)\to K(A,3)
$$
and the homology spectral sequence
$$
E_{p,q}^2=H_pK(A,3)\otimes H_qK(A,2)\Rightarrow \mathbb Z[0]
$$
The dual de Rham complexes can be recognized as natural parts of
the $E^3$-term of this spectral sequence. For example, we have the
following natural diagrams:
$$
\xyma{ E_{6,0}^3\ar@{->}[r]^{d_{6,0}^3}\ar@{=}[d] &
E_{3,2}^3\ar@{->}[r]^{d_{3,2}^3} \ar@{=}[d]& E_{0,4}^3\ar@{=}[d]\\
\Lambda^2(A) \ar@{->}[r]^{d_2} & A\otimes A\ar@{->}[r]^{d_1} &
\Gamma_2(A),}
$$
$$
\xyma{ E_{9,0}^3\ar@{->}[r]^{d_{9,0}^3} &
E_{6,2}^3\ar@{->}[r]^{d_{6,2}^3} \ar@{=}[d]& E_{3,5}^3\ar@{->}[r]^{d_{3,4}^3} \ar@{=}[d] & E_{0,6}^3 \ar@{=}[d]\\
\Lambda^3(A) \ar@{^{(}->}[u] \ar@{->}[r]^{d_3} &
\Lambda^2(A)\otimes A\ar@{->}[r]^{d_2} & A\otimes \Gamma_2(A)
\ar@{->}[r]^{d_1} & \Gamma_3(A)}
$$

\bigskip

We will  now  give a functorial description of  certain homology
groups of these  complexes $C^n(A)$. Some applications of these
results in the theory of derived functors one can find in
\cite{BM}.
\begin{prop}\label{jean}
Let $A$ be a free abelian. Then \begin{enumerate} \item
\cite{Franjou} For any prime number $p$, $H_0C^p(A)=A\otimes
\mathbb Z/p,$ and $H_iC^p(A)=0,$ for all $i>0$; \item \cite{Jean}
 There is a natural isomorphism
$$ H_0C^n(A)\simeq \bigoplus_{p|n,\ p\ \text{prime}}\Gamma_{n/p}(A\otimes \mathbb Z/p).
$$
\end{enumerate}
\end{prop}

We will make use of  the following fact from number theory (see
\cite{Loveless} corollary 2):
\begin{lemma}\label{number}
Let $n$ and $k$ be a pair positive integers and $p$ is a prime
number, then
$$
\binom{pn}{pk}\equiv \binom{n}{k}\mod p^r,
$$
where $r$ is the largest power of $p$ dividing $pnk(n-k)$.
\end{lemma}

\vspace{.5cm}\noindent{\it Proof of Proposition \ref{jean} (2).}
Let  $n\geq 2$  and define the map
$$
q_n: \Gamma_n(A)\to \bigoplus_{p|n}\Gamma_{n/p}(A\otimes \mathbb
Z/p)
$$
by setting:
\begin{multline*}
q_n: \gamma_{i_1}(a_1)\dots \gamma_{i_t}(a_t)\mapsto \sum_{p|i_k,\
\text{for all}\ 1\leq k\leq t}\gamma_{i_1/p}(\bar a_1)\dots
\gamma_{i_t/p}(\bar a_t),\\ i_1+\dots+i_t=n,\ a_k\in A,\ \bar
a_k\in A\otimes \mathbb Z/p.
\end{multline*}
If $(i_1,\dots, i_t)=1$, then we set
$$
q_n(\gamma_{i_1}(a_1)\dots \gamma_{i_t}(a_t))=0,\ \text{where}\
a_k\in A\ \text{for\ all}\ k.
$$
Let us check that the map $q_n$ is well-defined. For that we have
to show that
\begin{align}
& q_n(\gamma_{j_1}(x)\gamma_{j_2}(x)\dots
\gamma_{j_t}(x_t))=q_n(\binom{j_1+j_2}{j_1}\gamma_{j_1+j_2}(x)\dots
\gamma_{j_t}(x_t)) \label{q2}\\
& q_n(\gamma_{j_1}(x_1+y_1)\dots
\gamma_{j_t}(x_t))=\sum_{k+l={j_1}}q_n(\gamma_{k}(x_1)\gamma_l(y_1)\dots \gamma_{j_t}(x_t)) \label{q3}\\
& q_n(\gamma_{j_1}(-x_1)\dots
\gamma_{j_t}(x_t))=(-1)^{j_1}q_n(\gamma_{j_1}(x_1)\dots
\gamma_{j_t}(x_t))\label{q4}
\end{align}

Verification of (\ref{q2}). First suppose that $p|j_1+j_2,\ p\nmid
j_1$. Since for every pair of numbers $n\geq k$, one has
\begin{equation}\frac{n}{(n,k)}|\binom{n}{k},\end{equation}
we have $$ \binom{j_1+j_2}{j_1}=p^s m_1,\ j_1+j_2=p^s m_2,\
(m_1,p)=(m_2,p)=1,\ s\geq 1.
$$
Hence $\binom{j_1+j_2}{j_1}\frac{m_2}{m_1}=p(\frac{j_1+j_2}{p}).$
Observe that $\Gamma_{\frac{j_1+j_2}{p}}(A\otimes \mathbb Z/p)$ is
a
 $p$-group, hence
\begin{equation}\label{qa1}
\binom{j_1+j_2}{j_1}\gamma_{\frac{j_1+j_2}{p}}(\bar x)=0,\ \bar
x\in A\otimes \mathbb Z/p
\end{equation}
since
$$
\binom{j_1+j_2}{j_1}\frac{m_2}{m_1}\gamma_{\frac{j_1+j_2}{p}}(\bar
x)=p\bar x\gamma_{\frac{j_1+j_2}{p}-1}(\bar x)=0.
$$
The equality (\ref{qa1}) implies that
\begin{multline*}
q_n(\binom{j_1+j_2}{j_1}\gamma_{j_1+j_2}(x)\dots
\gamma_{j_t}(x_t))-q_n(\gamma_{j_1}(x)\gamma_{j_2}(x)\dots
\gamma_{j_t}(x_t))=\\ \binom{j_1+j_2}{j_1}\sum_{p|j_1+j_2,\
p|j_k,\ k>2}\gamma_{\frac{j_1+j_2}{p}}(\bar x)\dots
\gamma_{j_t/p}(\bar x_t)\, -\sum_{p|j_k,\ \text{for all}\ 1\leq
k\leq t}\gamma_{j_1/p}(\bar x)\gamma_{j_2/p}(\bar x)\dots
\gamma_{j_t/p}(\bar x_t)=\\ \left(
\binom{j_1+j_2}{j_1}-\binom{\frac{j_1+j_2}{p}}{\frac{j_1}{p}}\right)\sum_{p|j_k,\
\text{for all}\ 1\leq k\leq t}\gamma_{\frac{j_1+j_2}{p}}(\bar
x)\dots \gamma_{j_t/p}(\bar x_t)
\end{multline*}
Let $j_1+j_2=p^sm,\ (m,p)=1$. Lemma \ref{number} implies that
$$
\binom{j_1+j_2}{j_1}\equiv\binom{\frac{j_1+j_2}{p}}{\frac{j_1}{p}}\mod
p^r
$$
where $r$ is the largest power of $p$, dividing
$(j_1+j_2)j_1j_2/p^2.$ Since $p|j_1,\ p|j_2,$ we have
$$
\binom{j_1+j_2}{j_1}\equiv\binom{\frac{j_1+j_2}{p}}{\frac{j_1}{p}}\mod
p^s.
$$
Hence
$$
\left(\binom{j_1+j_2}{j_1}-\binom{\frac{j_1+j_2}{p}}{\frac{j_1}{p}}\right)\gamma_{\frac{j_1+j_2}{p}}(\bar
x)=0,\ \bar x\in A\otimes \mathbb Z/p
$$
and the property (\ref{q2}) follows.

\medskip

Verification of (\ref{q3}). We have

\begin{multline*}q_n(\gamma_{j_1}(x_1+y_1)\dots
\gamma_{j_t}(x_t))-\sum_{k+l={j_1}}q_n(\gamma_{k}(x_1)\gamma_l(y_1)\dots
\gamma_{j_t}(x_t))=\\ \sum_{p|j_t,\ t\geq 1}\gamma_{j_1/p} (\bar
x_1+\bar y_1)\dots \gamma_{j_t/p}(\bar
x_t)-\sum_{k+l=j_1}\sum_{p|k,p|l,p|j_t,\ t\geq 2}\gamma_{k/p}(\bar
x_1)\gamma_{l/p}(\bar y_1)\dots \gamma_{j_t/p}(\bar x_t)=\\
\sum_{p|j_t,\ t\geq 1}\gamma_{j_1/p} (\bar x_1+\bar y_1)\dots
\gamma_{j_t/p}(\bar x_t)-\sum_{p|k,p|l,p|j_t,\ t\geq
2}\sum_{\frac{k}{p}+\frac{l}{p}=\frac{j_1}{p}}\gamma_{k/p}(\bar
x_1)\gamma_{l/p}(\bar y_1)\dots \gamma_{j_t/p}(\bar x_t)=0
\end{multline*}

Verification of (\ref{q4}). We have
\begin{multline*}
q_n(\gamma_{j_1}(-x_1)\dots
\gamma_{j_t}(x_t))-(-1)^{j_1}q_n(\gamma_{j_1}(x_1)\dots
\gamma_{j_t}(x_t))=\\
\sum_{p|j_k,\ k\geq 1}(\gamma_{j_1/p}(-\bar
x_1)-(-1)^{j_1}\gamma_{j_1/p}(\bar x_1))\dots
\gamma_{j_t/p}(x_t)=0
\end{multline*}
since
$$
(\gamma_{j_1/p}(-\bar x_1)-(-1)^{j_1}\gamma_{j_1/p}(\bar x_1))=0
$$
(we separately check the cases $p=2$ and $p\neq 2$).

We  now know that the map $q_n$ is well-defined. It induces a  map
$$
\bar q_n: H^0C^n(A)\to \bigoplus_{p|n}\Gamma_{n/p}(A\otimes
\mathbb Z/p),
$$
since $q_n(a)=0$ for every $a\in \text{im}\{A\otimes
\Gamma_{n-1}(A)\to \Gamma_n(A)\}.$

Let $A=\mathbb Z$, then
$$
H_0(\mathbb Z)=\text{coker}\{\Gamma_{n-1}(\mathbb Z)\otimes
\mathbb Z\to \Gamma_n(\mathbb Z)\}\simeq \text{coker}\{\mathbb
Z\buildrel{n}\over\to \mathbb Z\}\simeq \mathbb Z/n.
$$
Let $n=\prod p_i^{s_i}$ be the prime decomposition of $n$. Then
$$
\bigoplus_{p_i|n}\Gamma_{n/p_i}(\mathbb
Z/p_i)=\bigoplus_{p_i|n}\mathbb Z/p_i^{s_i}=\mathbb Z/n.
$$
It follows from definition of the map $\bar q_n$, that
$$
\bar q_n: H_0C^n(\mathbb
Z)\to\bigoplus_{p_i|n}\Gamma_{n/p_i}(\mathbb Z/p_i)
$$
is an isomorphism.

For free abelian groups $A$ and $B$, one has a natural isomorphism
of complexes
$$
C^n(A\oplus B)\simeq \bigoplus_{i+j=n,\ i,j\geq 0} C^i(A)\otimes
C^j(B)
$$
This implies that the cross-effect\footnote{Given a functor
$F:{\sf Ab\to Ab}$, its cross effect is defined as the kernel of
the natural map $F(A|B)=ker\{F(A\oplus B)\to F(A)\oplus F(B)\}$,\
$A,B\in {\sf Ab}$} $C^n(A|B)$ of the functor $C^n(A)$ is described
by
$$
C^n(A|B)\simeq\bigoplus_{i+j=n,\ i,j> 0} C^i(A)\otimes C^j(B)
$$
and its homology $ (H_kC^n)(A|B)=H_kC^n(A|B) $ can be described
with the help of K\"{u}nneth formulas:
\begin{multline*}
0\to \bigoplus_{i+j=n,\ i,j>0,\ r+s=k}H_rC^i(A)\otimes
H_sC^j(B)\to H_kC^n(A|B)\to\\ \bigoplus_{i+j=n,\ i,j>0,\
r+s=k-1}\Tor(H_rC^i(A), H_sC^j(B))\to 0
\end{multline*}
Hence we have the following simple description of the cross-effect
of $H_0C^n(A)$:
\begin{equation}\label{dec1}
H_0C^n(A|B)\simeq \bigoplus_{i+j=n,\ i,j>0}H_0C^i(A)\otimes
H_0C^j(B)
\end{equation}
From the other hand, we have the following decomposition of the
cross-effect of the functor
$\tilde\Gamma_{n/p}^p(A):=\Gamma_{n/p}(A\otimes \mathbb Z/p)$:
\begin{equation*}\label{dec2a}
\tilde\Gamma_{n/p}^p(A|B)=\bigoplus_{l+k=n/p}\Gamma_{l}(A\otimes
\mathbb Z/p)\otimes \Gamma_k(B\otimes \mathbb Z/p)
\end{equation*}
Hence
\begin{equation}\label{dec2}
\bigoplus_{p|n}\tilde\Gamma_{n/p}^p(A|B)=\bigoplus_{p|n}\bigoplus_{l+k=n/p}\Gamma_{l}(A\otimes
\mathbb Z/p)\otimes \Gamma_k(B\otimes \mathbb Z/p)
\end{equation}
We must now show that the maps $\bar{q}_n$ preserve the
decompositions (\ref{dec1}) and (\ref{dec2}). This is equivalent
to the commutativity of the following diagram:

\begin{equation}\label{elementary}
\xymatrix{H_0C^{i}(A)\otimes H_0C^{j}(B) \ar@{->}[rr]^(.4){\bar
q_{i}\otimes
  \bar q_{j}}
\ar[d]^{\wr} & &\bigoplus_{p|i}\Gamma_{i/p}(A\otimes \mathbb
Z/p)\otimes \bigoplus_{p|j}\Gamma_{j/p}(B\otimes \mathbb Z/p)
\ar@{^{(}->}[dd]^{\varepsilon'}
\\  H_0(C^{i}(A)\otimes C^{j}(B))  \ar@{^{(}->}[d]
&&
\\  H_0C^{i+j}(A\oplus B) \ar[rr]_(.45){\bar q_{i+j}} & &
   \bigoplus_{p|(i+j)}\Gamma_{\frac{i+j}{p}}((A\oplus B)\otimes
   \mathbb Z/p)
}
\end{equation}

The map $\varepsilon'\circ(\bar q_{i}\otimes \bar q_{j})$ is
defined via the natural map
\begin{multline*}
\prod \gamma_{i_k}(x_k)\otimes \prod \gamma_{j_k}(y_k)\mapsto
\sum_{p|i_k,\ p|j_k}\prod \gamma_{i_k/p}(\bar x_i)\prod
\gamma_{j_k/p}(\bar y_i)\in
\bigoplus_{p|(i+j)}\Gamma_{\frac{i+j}{p}}((A\oplus B)\otimes
\mathbb Z/p),\\
x_k\in A,\ y_k\in B,\ \bar x_k\in A\otimes \mathbb Z/p,\ \bar
y_k\in B\otimes \mathbb Z/p,\ \sum j_k=j,\ \sum i_k=i
\end{multline*} and the commutativity of the diagram
(\ref{elementary}) follows. This proves that the  natural map
\begin{multline*}H_0C^n(A|B)\simeq
\bigoplus_{i+j=n,\ i,j>0}H_0C^i(A)\otimes H_0C^j(B)\to\\
\bigoplus_{p|n}\tilde\Gamma_{n/p}^p(A|B)=\bigoplus_{p|n}\bigoplus_{l+k=n/p}\Gamma_{l}(A\otimes
\mathbb Z/p)\otimes \Gamma_k(B\otimes \mathbb Z/p)
\end{multline*}
induced by $\bar{q}_n$ on cross-effects is an isomorphism, and it
follows from this that $\bar{q}_n$ is an isomorphism for all free
abelian groups $A$.
 $\Box$

\subsection{Derived functors and homology}

Let $A$ be an abelian group, and $F$  an endofunctor on the
category of abelian groups. Recall that for every  $n\geq 0$ the
derived functor of $F$
 in the sense
of Dold-Puppe  \cite{DoldPuppe} are defined by
$$
L_iF(A,n)=\pi_i(FKP_\ast[n]),\ i\geq 0
$$
where $P_\ast \to A$ is a projective resolution of $A$, and
 $K$ is  the Dold-Kan transform,  inverse to the Moore normalization  functor
\[
N:  \mathrm{Simpl}({\sf Ab}) \to C({\sf Ab})
\]
from simplicial abelian groups to chain complexes.

Recall the description of the highest derived functors of the
tensor power functor due to Mac Lance \cite{Mac1}. The group
$\Tor^{[n]}(A)$ is generated by the $n$-linear expressions
$\tau_h(a_1, \ldots a_n)$ (where all $a_i$ live in the subgroup $
{}_hA$ of elements $a$ of $A$ for which $ha=0  \ (h
>0)$, subject to the so-called slide relations
\begin{equation}
\label{slide} \tau_{hk}(a_1,\ldots, a_i,\ldots a_n) =
\tau_{h}(ka_1, \ldots, ka_{i-1}, a_i, k_{i+1}, \ldots, ka_n)
\end{equation} for all $i$ whenever  $hka_j = 0$ for all $j \neq
i$ and $ha_i=0$.
 The associativity of the derived
tensor product functor implies that there are   canonical
isomorphisms
$$
\Tor^{[n]}(A)\simeq \Tor(\Tor^{[n-1]}(A),A),\ n\geq 2.
$$
For $n\geq 2,$ there is a natural isomorphism:
$$
L_{n-1}\otimes^n(A)\simeq \Tor^{[n]}(A).
$$

The map $\ot^n \la SP^n$  induces  a natural epimorphism \bee
\label{epi} \Tor^{[n]}(A)\to L_{n-1}SP^n(A) \ee which sends the
generators $\tau_h(a_1,\dots,a_n)$ of $ \Tor^{[n]}(A)$ to
generators $\beta_h(a_1,\dots,a_n)$ of
\[ \mathcal{S}_n(A) :=L_{n-1}SP^n(A)\,.\]
The kernel of this map is generated by the elements
$\tau_h(a_1,\dots,a_n)$ with $a_i=a_j$ for some $i\neq j$. It is
shown by Jean  in  \cite{Jean} that
\begin{equation}\label{syder}
L_iSP^{n}(A)\simeq (L_iSP^{i+1}(A)\otimes
SP^{n-(i+1)}(A))/Jac_{SP},
\end{equation}
where $Jac_{SP}$ is the subgroup generated by elements of the form
$$
\sum_{k=1}^{i+2}(-1)^k\beta_h(x_1,\dots, \hat x_k,\dots,
x_{i+2})\otimes x_ky_1\dots y_{n-i-2}.
$$ with $x_i \in {}_hA$ and $y_j \in A$ for all $i,j$.

We will now construct a series of maps:
$$
f_i^{n,p}: L_iSP^{n/p}(A\otimes \mathbb Z/p)\to H_iC^n(A)
$$
for a free abelian $A$ and $p|n$. We first  choose liftings
$(x_i)$ to  $A$
 of a given  family of elements
$(\bar{x}_i) \in A \otimes \mathbb{Z}_p$.  We set
\begin{multline*}
f_i^{n,p}:\beta_p(\bar x_1,\dots, \bar x_{i+1})\otimes \bar
x_{i+2}\dots
\bar x_{\frac{n}{p}}\mapsto \eta_i(x_1,\dots,x_{\frac{n}{p}}):=\\
\sum_{t=1}^{i+1}(-1)^tx_{1}\wedge \dots \wedge \hat x_t\wedge
\dots \wedge x_{i+1}\otimes \gamma_{p-1}(x_1)\dots
\widehat{\gamma_{p-1}(x_{t})}\dots\gamma_{p-1}(x_{i+1})\gamma_p(x_t)\gamma_p(x_{i+2})\dots
\gamma_p(x_{\frac{n}{p}})\\
\in \Lambda^i(A)\otimes \Gamma_{n-i}(A),\ \bar x_k\in A\otimes
\mathbb Z/p,\ x_k\in A.
\end{multline*}
\begin{prop}
The maps $f_i^{n,p}$ are well defined for all $i,n,p$.
\end{prop}
\begin{proof}
We have
\begin{multline*}
\eta_i(px_1,\dots,x_{\frac{n}{p}})=\\
\sum_{t=1}^{i+1}(-1)^tpx_{1}\wedge \dots \wedge \hat x_t\wedge
\dots \wedge x_{i+1}\otimes \gamma_{p-1}(x_1)\dots
\widehat{\gamma_{p-1}(x_{t})}\dots\gamma_{p-1}(x_{i+1})\gamma_p(x_t)\gamma_p(x_{i+2})\dots
\gamma_p(x_{\frac{n}{p}})\\
\sum_{t=1}^{i+1}(-1)^tx_{1}\wedge \dots \wedge \hat x_t\wedge
\dots \wedge x_{i+1}\otimes x_t\gamma_{p-1}(x_1)\dots
\gamma_{p-1}(x_{t})\dots\gamma_{p-1}(x_{i+1})\gamma_p(x_{i+2})\dots
\gamma_p(x_{\frac{n}{p}})=\\
d_{i+1}(x_1\wedge \dots \wedge x_{i+1}\otimes
\gamma_{p-1}(x_1)\dots
\gamma_{p-1}(x_{t})\dots\gamma_{p-1}(x_{i+1})\gamma_p(x_{i+2})\dots
\gamma_p(x_{\frac{n}{p}}))\\
\in \text{im}\{\Lambda^{i+1}(A)\otimes
\Gamma_{n-i-1}(A)\buildrel{d_{i+1}}\over\to \Lambda^i(A)\otimes
\Gamma_{n-i}(A)\}
\end{multline*}
One verifies that  for every $1\leq k\leq n/p$, one has
$$
\eta_i(x_1,\dots, px_k,\dots, x_{\frac{n}{p}})\in
im\{\Lambda^{i+1}(A)\otimes
\Gamma_{n-i-1}(A)\buildrel{d_{i+1}}\over\to \Lambda^i(A)\otimes
\Gamma_{n-i}(A)\}
$$
It follows that the map $f_i^{n,p}:  L_iSP^{n/p}(A\otimes \mathbb
Z/p)\to (\Lambda^i(A)\otimes \Gamma_{n-i}(A))/\text{im}(d_{i+1})$
is well-defined.The simplest examples of such elements are the
following
\begin{align*}
& \eta_1(x_1,x_2)=x_1\otimes x_1\gamma_2(x_2)-x_2\otimes x_2\gamma_2(x_1)\in A\otimes \Gamma_2(A),\\
& \eta_2(x_1,x_2,x_3)=-x_1\wedge x_2\otimes
x_1x_2\gamma_2(x_3)+x_1\wedge x_3\otimes x_1x_3\gamma_2(x_2)\\
& \ \ \ \ \ \ \ \ \ \ \ \ \ \ -x_2\wedge x_3\otimes
x_2x_3\gamma_2(x_1)\in \Lambda^2(A)\otimes \Gamma_4(A)
\end{align*}
By construction, the elements $\eta_i$ lie  in
$\Lambda^i(A)\otimes \Gamma_{n-i}(A).$ In fact, let us verify that
$$
\eta_i(x_1,\dots,x_{\frac{n}{p}})\in \ker\{\Lambda^i(A)\otimes
\Gamma_{n-i}(A)\buildrel{d_i}\over\to \Lambda^{i-1}(A)\otimes
\Gamma_{n-i+1}(A)\}
$$
Observe that
\begin{multline*}
d_i\eta_i(x_1,\dots,x_{\frac{n}{p}})=\\
d_i\sum_{t=1}^{i+1}(-1)^tx_{1}\wedge \dots \wedge \hat x_t\wedge
\dots \wedge x_{i+1}\otimes \gamma_{p-1}(x_1)\dots
\widehat{\gamma_{p-1}(x_{t})}\dots\gamma_{p-1}(x_{i+1})\gamma_p(x_t)\gamma_p(x_{i+2})\dots
\gamma_p(x_{\frac{n}{p}})
\end{multline*}
In this sum, for every pair of indexes $1\leq r<s\leq i+1$ there
occurs a pair of  terms
$$
(-1)^{s+r}x_{1}\wedge \dots  \hat x_r\dots \hat x_s  \dots \wedge
x_{i+1}\otimes x_r\gamma_{p-1}(x_1)\dots
\widehat{\gamma_{p-1}(x_{s})}\dots\gamma_{p-1}(x_{i+1})\gamma_p(x_s)\gamma_p(x_{i+2})\dots
\gamma_p(x_{\frac{n}{p}})
$$
and
$$
(-1)^{s+r+1}x_{1}\wedge \dots  \hat x_r\dots \hat x_s  \dots
\wedge x_{i+1}\otimes x_s\gamma_{p-1}(x_1)\dots
\widehat{\gamma_{p-1}(x_{r})}\dots\gamma_{p-1}(x_{i+1})\gamma_p(x_r)\gamma_p(x_{i+2})\dots
\gamma_p(x_{\frac{n}{p}})
$$
which cancel each other. It follows that  the entire sum is equal
to zero.

For the same reason, the map
\begin{multline*}
\sum_{k=1}^{i+2}(-1)^k\beta_p(x_1,\dots, \hat x_k,\dots,
x_{i+2})\otimes x_ky_1\dots y_l\mapsto\\
\sum_{k=1}^{i+2}(-1)^k\sum_{m=1,\ m<k}^{i+2}(-1)^m x_1\wedge \dots
\hat x_m \dots \hat x_k\dots\wedge x_{i+2}\otimes\\ (\prod_{l=1,\
l\neq m,k}^{i+2}\gamma_{p-1}(x_l))
\gamma_{p}(x_{m})\gamma_p(x_k)\gamma_p(y_1)\dots \gamma_p(y_l)+\\
\sum_{k=1}^{i+2}(-1)^k\sum_{m=1,\ m>k}^{i+2}(-1)^{m+1} x_1\wedge
\dots \hat x_k \dots \hat x_m\dots\wedge x_{i+2}\otimes\\
(\prod_{l=1,\ l\neq m,k}^{i+2}\gamma_{p-1}(x_l))
\gamma_{p}(x_{m})\gamma_p(x_k)\gamma_p(y_1)\dots \gamma_p(y_l)
\end{multline*}
is trivial.
\end{proof}

Given abelian group $A$ and $i>0$, consider a natural map
$$
f_i^n=\sum_{p|n}f_i^{n,p}: \bigoplus_{p|n,\ p\
\text{prime}}L_iSP^{\frac{n}{p}}(A\otimes \mathbb Z/p)\to
H_iC^n(A)
$$

\begin{theorem}\label{theoremdrham}
The map $f_i^n$ is an isomorphism for $i>0,\ n\leq 7$.
\end{theorem}

\begin{proof}
Given an abelian simplicial group $(G_\bullet,
\partial_i,s_i)$, let $A(G_\bullet)$ be the  associated chain complex
with $A(G_\bullet)_n=G_n$, $d_n=\sum_{i=0}^n(-1)^n\partial_i$.
Recall that given abelian simplicial groups $G_\bullet$ and
$H_\bullet$, the Eilenberg-Mac Lane map
$$
g: A(G_\bullet)\otimes A(H_\bullet)\to A(G_\bullet\otimes
H_\bullet)
$$
is given by
$$
g(a_p\otimes b_q)=\sum_{(p;q)-\text{shuffles}\
(\mu,\nu)}(-1)^{\text{sign}(\mu,\nu)}s_{\nu_q}\dots
s_{\nu_1}(a_p)\otimes s_{\mu_p}\dots s_{\mu_1}(b_q)
$$
For any  free abelian group $A$ and  infinite cyclic group $B$, we
will show that there is a natural commutative diagram with
vertical K\"{u}nneth short exact sequences:
\begin{equation}\label{derhamdia}
 \xymatrix@C= 8pt{\bigoplus_{i+j=\frac{n}{p},\ r+s=k}L_rSP^i(A\otimes
\mathbb Z/p)\otimes L_sSP^j(B\otimes \mathbb Z/p) \ar@{^{(}->}[d]
\ar@{->}[r] & \bigoplus_{i+j=n,\
r+s=k}H_rC^i(A)\otimes H_sC^j(B) \ar@{^{(}->}[d]\\
L_kSP^{\frac{n}{p}}((A\oplus B)\otimes \mathbb Z/p) \ar@{->>}[d]
\ar@{->}[r]^{f_k^{n,p}} & H_kC^n(A\oplus B)
\ar@{->>}[d]\\
\bigoplus_{i+j=\frac{n}{p},\ r+s=k-1}\Tor(L_rSP^i(A), L_sSP^j(B))
\ar@{->}[r] & \bigoplus_{i+j=n,\ r+s=k-1} \Tor(H_rC^i(A),
H_sC^j(B))}
\end{equation}
where all maps are induced by maps $f_*^{*,p}$. Since $B$ is a
cyclic, it is enough to consider the case $s=0$ and summands of
the upper square from (\ref{derhamdia})
\begin{equation}\label{sd22}
\xyma{L_rSP^i(A\otimes \mathbb Z/p)\otimes SP^j(B\otimes \mathbb
Z/p) \ar@{->}[r]^{f_r^{ip,p}\otimes f_0^{jp,p}}
\ar@{^{(}->}[d]^{s_r} & H_rC^{ip}(A\otimes \mathbb Z/p)\otimes
H_0C^{jp}(B\otimes \mathbb Z/p)\ar@{^{(}->}[d]^{h_r}\\
L_rSP^{i+j}((A\oplus B)\otimes \mathbb Z/p)
\ar@{->}[r]^{f_r^{ip+jp,p}} & H_rC^{ip+jp}((A\oplus B)\otimes
\mathbb Z/p)\, ,} \end{equation} where the maps $s_r,h_r$  come
from K\"{u}nneth exact sequences. Consider natural projections
\begin{align*}
& u_1: L_rSP^{r+1}(A\otimes \mathbb Z/p)\otimes
SP^{i-r-1}(A\otimes
\mathbb Z/p)\to L_rSP^i(A\otimes \mathbb Z/p)\\
& u_2: L_rSP^{r+1}((A\otimes B)\otimes \mathbb Z/p)\otimes
SP^{i+j-r-1}((A\oplus B)\otimes \mathbb Z/p)\to
L_rSP^{i+j}((A\oplus B)\otimes \mathbb Z/p)
\end{align*}
The map $s_r$ is defined  by \begin{multline*} s_r:
u_1(\beta_p(a_1,\dots, a_{r+1})\otimes a_{r+2}\dots a_i)\otimes
b_1\dots b_j\mapsto u_2(\beta_p(a_1,\dots,a_{r+2})\otimes
a_{r+2}\dots a_ib_1\dots b_j),\\ a_k\in A\otimes \mathbb Z/p,\
b_l\in B\otimes \mathbb Z/p
\end{multline*}
We have
$$
f_r^{ip,p}\otimes f_0^{jp,p}(u_1(\beta_p(a_1,\dots,
a_{r+1})\otimes a_{r+2}\dots a_i)\otimes b_1\dots
b_j)=\eta_r(a_1,\dots,a_i)\otimes \gamma_p(b_1)\dots\gamma_p(b_j)
$$
and we see that the diagram (\ref{sd22}) is commutative.

The  map $f_i^n$ is an isomorphism for a cyclic group $A$, since
both source and target groups are trivial. For $i=1, n=4,$ and
cyclic $B$, we have a natural diagram
$$
\xyma{L_1SP^2(A\otimes \mathbb Z/2|B\otimes \mathbb
Z/2)\ar@{->}[r] & H_1C^4(A|B)\\
\tor(A\otimes \mathbb Z/2,B\otimes \mathbb Z/2) \ar@{->}[u]^{\wr}
\ar@{->}[r] & \tor(H_0C^2(A),H_0C^2(B)) \ar@{->}[u]^{\wr}}
$$
and the isomorphism $f_1^4$ follows.  The proof is similar for
other $i,n$. The only non-trivial case here is $i=1,n=6$, for the
2-torsion component of $H_1C^6(A)$. In that  case,  the statement
follows from the natural isomorphism
$$
\tor(SP^2(A\otimes \mathbb Z/2),\mathbb Z/2)\to
\tor(\Gamma_2(A\otimes \mathbb Z/2),\mathbb Z/2)
$$
for every free abelian group $A$.
\end{proof}

\begin{remark}
For any free abelian group $A$ and prime number  $p$, there are
canonical isomorphisms
\begin{align*}
& L_{n-1}SP^n(A\otimes \mathbb Z/p)\simeq \Lambda^n(A\otimes \mathbb Z/p)\\
& L_iSP^n(A\otimes \mathbb
Z/p)=\mathrm{coker}\{\Lambda^{i+2}(A\otimes \mathbb Z/p)\otimes
SP^{n-i-2}(A\otimes \mathbb Z/p)\buildrel{\kappa_{i+2}}\over\to\\
& \ \ \ \ \ \ \ \ \ \ \ \ \ \ \ \ \ \ \ \ \ \ \ \ \ \ \ \ \ \ \ \
\ \ \ \ \ \ \ \ \ \ \ \ \ \ \ \ \ \  \Lambda^{i+1}(A\otimes
\mathbb Z/p)\otimes SP^{n-i-1}(A\otimes \mathbb Z/p)\}
\end{align*}
where $\kappa_{i}$ is the corresponding differential in the $n$-th
Koszul complex.
\end{remark}

\noindent When $i=1,n=3$ and $A$ a free abelian group there is a
natural isomorphism
$$
L_1SP^3(A\otimes \mathbb Z/p)\simeq \EuScript L^3(A\otimes \mathbb
Z/p),
$$
where $\EuScript L^3$ is the third Lie functor (see \cite{BM}, for
example). Observe however that the natural map
$$
f_1^8: L_1SP^4(A\otimes\mathbb Z/2)\to H_1C^8(A)
$$
is not an isomorphism. Indeed, every element of
$L_1SP^4(A\otimes\mathbb Z/2)$ is
  2-torsion, whereas  $H_1C^8(A)$ can contain  4-torsion
elements, since its cross-effect $H_1C^8(A|B)$ contains
$\Tor(\Gamma_2(A\otimes\mathbb Z/2),\Gamma_2(B\otimes \mathbb
Z/2))$  as a subgroup. The map $f_1^8$ is given by
$$
\beta_2(\bar a,\bar b)\otimes \bar c\bar d\mapsto a\otimes
a\gamma_2(b)\gamma_2(c)\gamma_2(d)-b\otimes
b\gamma_2(a)\gamma_2(c)\gamma_2(d),\ a,b,c,d\in A.
$$

The following table, which is a consequence  theorem
\ref{theoremdrham}, gives a  complete description of $H_iC^n(A)$
for $n\leq 7$ and $A$   free abelian:
$$
\begin{tabular}{ccccccccccccccc}
 $q$ & \vline & $H_0C^q(A)$ & $H_1C^q(A)$ & $H_2C^q(A)$ & $H_3C^q(A)$\\
 \hline
7 & \vline & $A\otimes \mathbb Z/7$ & 0 & 0 & 0\\
6 & \vline & $\Gamma_2(A\otimes \mathbb Z/3)\oplus \Gamma_3(A\otimes\mathbb Z/2)$ & $\Lambda^2(A\otimes \mathbb Z/3)\oplus \EuScript L^3(A\otimes \mathbb Z/2)$ & $\Lambda^3(A\otimes \mathbb Z/2)$ & 0\\
5 & \vline & $A\otimes \mathbb Z/5$ & 0 & 0 & 0\\
4 & \vline & $\Gamma_2(A\otimes \mathbb Z/2)$ & $\Lambda^2(A\otimes \mathbb Z/2)$ & 0 & 0\\
3 & \vline & $A\otimes \mathbb Z/3$ & 0 & 0 & 0\\
2 & \vline & $A\otimes \mathbb Z/2$ & 0 & 0 & 0\\
\end{tabular}
$$
\noindent For example, the isomorphism
 \bee\label{c4iso}
f: \Lambda^2(A\otimes \mathbb Z/2)\to H_1C^4(A)\ee is defined, for
representatives $a,b\in A$ of $ \bar a,\bar b\in A\otimes \mathbb
Z/2  $, by
$$
f: \bar a\otimes \bar b\mapsto a\otimes a\gamma_2(b)-b\otimes
b\gamma_2(a).
$$

\vspace{.5cm} \noindent{\it Acknowledgement.} The author thanks L.
Breen for various discussions related to the subject of the paper.

\end{document}